\documentclass{article}
\usepackage{amsmath,amssymb,amsthm}
\usepackage{amscd}
\usepackage[dvips, dvipdfmx]{graphicx}\usepackage{color}

\newtheorem{thm}{Theorem}[section]
\newtheorem{lemma}[thm]{Lemma}

\newtheorem{prop}[thm]{Proposition}

\newtheorem{cor}[thm]{Corollary}

\newcommand{\va}{{\bf a}}
\newcommand{\vb}{{\bf b}}
\newcommand{\vc}{{\bf c}}

\newcommand{\Z}{\mathbb Z}
\newcommand{\Q}{\mathbb Q}

\newcommand{\hZ}{\widehat{\Z}}

\newcommand{\Aut}{{\rm Aut}}
\newcommand{\End}{{\rm End}}

\newcommand{\diag}{{\rm diag}}

\newcommand{\bs}{\backslash}

\newcommand{\Imat}{\diag[1,1]}

\newcommand{\Hei}{\mathcal{H}}
\newcommand{\hl}{\rule{0ex}{3ex}\rule{0ex}{3ex}\widehat{L}}

\newcommand{\halpha}{\rule{0ex}{3ex}\widehat{\alpha}}
\newcommand{\hbeta}{\rule{0ex}{3ex}\widehat{\beta}}

\newcommand{\rr}{R_r}
\newcommand{\dr}{\Delta_r}
\newcommand{\gamr}{\Gamma_r}

\newcommand{\trpi}{T_{r,p}^{(i)}}

\newcommand{\rrp}{R_{r,p}}
\newcommand{\drp}{\Delta_{r,p}}
\newcommand{\gamrp}{\Gamma_{r,p}}

\newcommand{\gl}{G_L}
\newcommand{\dl}{\Delta_L}
\newcommand{\gaml}{\Gamma_L}

\newcommand{\glp}{G_{L_p}}
\newcommand{\dlp}{\Delta_{L_p}}
\newcommand{\gamlp}{\Gamma_{L_p}}

\newcommand{\hgl}{\rule{0ex}{3ex}\widehat{G}_L}
\newcommand{\hdl}{\rule{0ex}{3ex}\widehat{\Delta}_L}
\newcommand{\hgaml}{\rule{0ex}{3ex}\widehat{\Gamma}_L}

\newcommand{\etal}{\eta_L}
\newcommand{\etadl}{\eta_{*L}}
\newcommand{\etaul}{\eta_L^*}

\newcommand{\rl}{R_L}
\newcommand{\rlp}{R_{L_p}}
\newcommand{\hrl}{\rule{0ex}{3ex}\widehat{R}_L}

\newcommand{\gH}{G_\mathcal{H}}
\newcommand{\dH}{\Delta_\mathcal{H}}
\newcommand{\gamH}{\Gamma_\mathcal{H}}

\newcommand{\dHp}{\Delta_{\mathcal{H}_p}}
\newcommand{\gamHp}{\Gamma_{\mathcal{H}_p}}

\newcommand{\hdH}{\rule{0ex}{3ex}\widehat{\Delta}_\mathcal{H}}
\newcommand{\hgamH}{\rule{0ex}{3ex}\widehat{\Gamma}_\mathcal{H}}

\newcommand{\etaH}{\eta_\mathcal{H}}
\newcommand{\etadH}{\eta_{*\mathcal{H}}}
\newcommand{\etauH}{\eta_\mathcal{H}^*}

\newcommand{\rH}{R_\mathcal{H}}
\newcommand{\rHp}{R_{\mathcal{H}_p}}
\newcommand{\hrH}{\rule{0ex}{3ex}\widehat{R}_\mathcal{H}}

\newcommand{\dZr}{\Delta_{\Z^r}}
\newcommand{\gamZr}{\Gamma_{\Z^r}}

\newcommand{\hdZr}{\rule{0ex}{3ex}\widehat{\Delta}_{\Z^r}}
\newcommand{\hgamZr}{\rule{0ex}{3ex}\widehat{\Gamma}_{\Z^r}}

\newcommand{\etadZr}{\eta_{*{\Z^r}}}

\newcommand{\gzpk}{\Gamma_0(p^k)}
\newcommand{\gplpk}{\Gamma(p^l,p^{l+k})}

\newcommand{\glk}{G_{l,k}}
\newcommand{\glko}{G_{l,k}^{1}}
\newcommand{\glkpm}{G_{l,k}^{\pm}}

\newcommand{\eqna}[1]{
	\begin{eqnarray*}
		#1
	\end{eqnarray*}
}

\newcommand{\pmat}[4]{
\begin{pmatrix}
#1 & #2 \\ #3 & #4
\end{pmatrix}
}

\begin{document}
\title{A note on a Hecke ring associated with the Heisenberg Lie algebra}


\author{Fumitake Hyodo}



\date{}

\maketitle

\begin{abstract}

This paper focuses on the theory of the Hecke rings associated with the general linear groups originally studied by Hecke and Shimura et al., and moreover generalizes its notions to Hecke rings associated with the automorphism groups of certain algebras. Then, in the case of the Heisenberg Lie algebra, we show an analog of the classical theory.
\end{abstract}

\section{Introduction}
\quad Hecke rings were introduced by Shimura \cite{S1}, which generalize the ring of Hecke operators studied by Hecke \cite{H}.
Later, in Chapter 3 of his book \cite{S2}, Shimura improved the definition of Hecke rings to a purely algebraic one, and also introduced a part of the theory of Tamagawa \cite{T} on the Hecke rings associated with the general linear groups over $\Q$. The Hecke rings treated there can be divided into local and global ones. 
%
Let $r$ be a positive integer, and let $p$ be a prime number.
The former is the Hecke ring $\rrp$ with respect to the pair $\left(GL_r(\Z_p), GL_r(\Q_p) \cap M_r(\Z_p)\right)$. It is the polynomial ring over $\Z$ on $r$ variables (cf. Theorem \ref{thm:local:Hecke-Tamagawa}).
The latter is the Hecke ring $\rr$ with respect to the pair $(GL_r(\Z), GL_r(\Q) \cap M_r(\Z))$, which is related to the local Hecke rings $\rrp$ in the following way (cf. Theorem \ref{thm:Hecke-Tamagawa}).
\begin{enumerate}
\item For each $p$, the local Hecke ring $\rrp$ embeds naturally into the global one $\rr$.
\item The local Hecke rings $\rrp$ commute with each other in $\rr$.
\item $\rr$ is generated by the local Hecke rings $\rrp$ as a ring.
\end{enumerate} 

We generalize the Hecke rings $\rr$ and $\rrp$. Namely, for a prime number $p$ and a certain algebra $L$, we deal with the local Hecke ring $\rlp$ and the global Hecke ring $\rl$ respectively with respect to the pair 
\[\left( \Aut_{\Z_p}^{alg}(L\otimes \Z_p), \ \End_{\Z_p}^{alg}(L\otimes \Z_p) \cap \Aut_{\Q_p}^{alg}(L\otimes \Q_p)\right),\]
 and the pair 
 \[\left( \Aut_{\Z}^{alg}(L), \ \End_{\Z}^{alg}(L) \cap \Aut_{\Q}^{alg}(L\otimes \Q)\right).\] 

Our previous paper \cite{Hy} studied the local case of the Heisenberg Lie algebra $\Hei$, that is, the quotient of the free Lie algebra $L_2$ on two generators over $\Z$ by the ideal $ [L_2 ,[L_2 ,L_2 ]]$.
As a result, it was proved that the local Hecke ring $\rHp$ is noncommutative, unlike the Hecke rings $\rrp$.

In this paper, we discuss the global Hecke ring $\rH$. First, in Section \ref{section:profin}, we introduce a Hecke ring $\hrl$ associated with the profinite completion $\hl$ of $L$, which satisfies the properties similar to $\rr$ (cf. Proposition \ref{prop:structureofhrl}).
Then, we construct an additive map $\etaul:\hrl \to \rl$. Furthermore, the multiplicativity and the injectivity are proved if $L$ is the Heisenberg Lie algebra $\Hei$.
Finally, in Section \ref{section:nonsurjectivity}, we prove the nonsurjectivity of $\etauH$, and conclude with the following result, which is analogous to the properties of the Hecke ring $\rr$ except for assertion \ref{assertion:nonsurjectivity} (cf. Theorem \ref{thm:main}):

\begin{thm}[Main Theorem]
The following assertions hold.
\begin{enumerate}
\item For each $p$, the local Hecke ring $\rHp$ is embedded into $\rH$ by $\etauH$.\label{assertion:embedding}
\item The local Hecke rings $\rHp$ commute with each other in $\rH$. \label{assertion:commutativity}
\item $\rH$ is \textbf{\textit{not}} generated by the local Hecke rings $\rHp$ as a ring. \label{assertion:nonsurjectivity}
\end{enumerate}
\end{thm}

It is the most essential to show assertion \ref{assertion:nonsurjectivity}, which is equivalent to the nonsurjectivity of $\etauH$. It is settled by calculating the cardinalities of the $\gamH$-double cosets of certain subsets of $\dH$ (cf. Corollary \ref{cor:explicitcaluculation}). 

It is worth pointing out that we treat a new type of noncommutative Hecke rings. While the Hecke rings associated with classical groups due to Andrianov \cite{A1}, Hina--Sugano \cite{Hi}, Satake \cite{Sa}, and Shimura \cite{S3} are all commutative, the local Hecke rings $\rHp$ are all noncommutative as mentioned before, and so is the global Hecke ring $\rH$ containing them. Although Dulinsky \cite{D} and Iwahori--Matsumoto \cite{IM} studied certain noncommutative Hecke rings, our Hecke rings $\rl$ and $\rlp$ are different from them in general. 

Further study of the Hecke rings $\rl$ and $\rlp$ is now in progress by the author. 
In \cite{Hy2}, we studied the local Hecke rings associated with the higher Heisenberg Lie algebras. In \cite{Hy3}, we investigated formal Dirichlet series with coefficients in $\rl$ that replaces Shimura's series with coefficients in $\rr$ \cite{S2}.

\section{Preliminary}\label{section:preliminary}\quad 
In this section, we recall the Hecke rings in \cite[Chapter 3]{S2}, and define our Hecke rings.
Let us fix a positive integer $r$ and a prime number $p$. We put $\gamrp = GL_r(\Z_p)$ and $\drp=GL_r(\Q_p) \cap M_r(\Z_p)$. Denote by $\rrp$ the Hecke ring  with respect to the pair $(\gamrp, \drp)$. Hecke and Tamagawa proved the following theorem.

\begin{thm}[{\cite{H}} and {\cite{T}}]\label{thm:local:Hecke-Tamagawa}
Let us put \[\trpi = \gamrp \diag[1,...,1,\overbrace{p,...,p}^i] \gamrp \in \rrp, \] for each $i$ with $1 \leq i \leq r$. Then $\rrp$ is the polynomial ring over $\Z$ on variables $\trpi$ with $1 \leq i \leq r$.
\end{thm} 
\medskip

Let us put $\gamr=GL_r(\Z)$ and $\dr = M_r(\Z)\cap GL_r(\Q)$, and denote by $\rr$ the Hecke ring with respect to $(\gamr, \dr)$.
By \cite[Proposition 3.16]{S2} and the elementary divisor theorem, we have the following theorem.
%
\begin{thm}\label{thm:Hecke-Tamagawa}
	The following assertions hold.
	\begin{enumerate}
		\item For each $p$, the local Hecke ring $\rrp$ is embedded naturally into the global one $\rr$.
		\item The local Hecke rings $\rrp$ commute with each other in $\rr$.
		\item $\rr$ is generated by the local Hecke rings $\rrp$ as a ring.
	\end{enumerate}
\end{thm}
\bigskip

We next define our Hecke rings. Throughout this paper, by an algebra we mean an abelian group with a bi-additive product (e.g., an associative algebra, a Lie algebra).
Let $L$ be an algebra which is free of rank $r$ as an abelian group, and fix a $\Z$-basis of $L$. Then $\Aut_{\Q}^{alg}(L \otimes \Q)$, $\Aut_{\Q_p}^{alg}(L \otimes \Q_p)$, $\End_{\Z}^{alg}(L)$, and $ \End_{\Z_p}^{alg}(L \otimes \Z_p)$ are all identified with subsets of $M_r(\Q_p)$. We introduce the following notation.
\[
	\begin{array}{llllll}
		\gl &=& \Aut_{\Q}^{alg}(L \otimes \Q), & \glp &=& \Aut_{\Q_p}^{alg}(L \otimes \Q_p), \\ \rule{0ex}{3ex}
		 \dl &=& \End_{\Z}^{alg}(L) \cap \gl, & \dlp &=& \End_{\Z_p}^{alg}(L \otimes \Z_p) \cap \glp, \\ \rule{0ex}{3ex}
		 \gaml  &=& \Aut_{\Z}^{alg}(L), & \gamlp  &=& \Aut_{\Z_p}^{alg}(L \otimes \Z_p).
	\end{array}
\]
By \cite[Proposition 2.1]{Hy}, one can define the Hecke rings $\rl$ and $\rlp$ with respect to $(\gaml, \dl)$ and $(\gamlp,\dlp)$, respectively.
Although the proposition deals only with the case where $L$ is a Lie algebra, its proof can be easily applied to our case as well. 
Our Hecke rings indeed generalize the Hecke rings $\rr$ and $\rrp$: If $L$ is the ring $\Z^r$ of the direct sum of $r$-copies of $\Z$, then we have $\dl = \dr, \gamlp = \gamrp, \rl = \rr$, and $\rlp = \rrp$. 

%
\section{Hecke rings associated with profinite completions of algebras}\label{section:profin}\quad 

Let $L$ be as in the previous section. First, we introduce a Hecke ring associated with the profinite completion $\hl$ of $L$. Let us define $\hgl$ by the set of elements $(\alpha_p)_p$
of $\prod_p \glp$ such that $\alpha_p \in \gamlp$ for almost all $p$, where the product is taken over all prime numbers $p$. $\hdl$ and $\hgaml$ denote $\left(\prod_p \dlp\right) \cap \hgl$ 
and $\prod_p \gamlp$, respectively.
Then it is easy to see that the $\hdl$ is contained in the commensurator of $\hgaml$ in $\hgl$. Hence, one can define the Hecke ring $\hrl$ with respect to $(\hgaml, \hdl)$. 
Since the monoids $\dlp$ are all naturally contained in $\hdl$ and commute with each other in $\hdl$, the Hecke ring $\hrl$ satisfies the following properties.
\begin{prop}\label{prop:structureofhrl} The following assertions hold.
\begin{enumerate}
\item For each $p$, the local Hecke ring $\rlp$ is embedded into $\hrl$ by the map derived by the natural inclusion $\dlp \subset \hdl$.
\item The local Hecke rings $\rlp$ commute with each other in $\hrl$.
\item $\hrl$ is generated by the local Hecke rings $\rlp$ as a ring.
\end{enumerate} 
\end{prop}

We next construct an additive map $\etaul:\hrl \to \rl$. 
Let us fix a $\Z$-basis of $L$, and let $p$ be a prime number. Then $\dl$ and $\dlp$ are identified with subsets of $M_r(\Q)$ and $M_r(\Q_p)$, respectively. Moreover, in this sense, $\dl$ is contained in $\dlp$.
The map $\etal$ then denotes the diagonal embedding of $\dl$ into $\prod_p \dlp$. Clearly, $\etal(\dl)$ is a subset of $\hdl$.
Then, we define the additive map $\eta^*_L : \hrl \to \rl$ by $\left(\hgaml \halpha \hgaml \mapsto \sum \gaml \beta \gaml\right)$, where $\gaml \beta \gaml$ runs through the set $\gaml \bs \etal^{-1}(\hgaml \halpha \hgaml) / \gaml$. 
%
%
Note that its kernel is the $\Z$-submodule of $\hrl$ generated by elements $\hgaml \halpha \hgaml$ with $\etal^{-1}(\hgaml \halpha \hgaml) = \emptyset$.

Next, we explain a relation between $\hrl$ and $\hl$: 
Since $\hl$ is isomorphic to $\prod_p (L \otimes \Z_p)$, the group $\hgaml$ is naturally isomorphic to $\Aut_{\hZ}^{alg}(\hl)$. Moreover,
the monoid $\hdl$ is identified with the set of elements $\halpha$ of $\End_{\hZ}^{alg}(\hl)$ such that $\hl^{\halpha}$ are $\hZ$-subalgebras of $\hl$ of finite index isomorphic to $\hl$. Hence, $\etal$ coincides with the restriction to $\dl$ of the map from $\End_{\Z}^{alg}(L)$ into $\End_{\hZ}^{alg}(\hl)$ induced by extension of scalars along the natural inclusion $\Z \to \hZ$.

Furthermore, by considering a $\Z$-basis of $L$, the monoid $\hdl$ and the group $\hgaml$ are identified with a submonoid of $M_r(\hZ)$ and a subgroup of $GL_r(\hZ)$, respectively. In this sense, $\etal$ is nothing but the natural inclusion.
\smallskip

Next, we give a sufficient condition for the multiplicativity and the injectivity of $\etaul$. Let us denote by $\etadl: \gaml \bs \dl \to \hgaml \bs \hdl$ the map induced by $\etal$.

\begin{lemma}\label{lemma:bijectivityandmultiplicative}
If the map $\etadl$ is bijective, then $\etaul$ is multiplicative and injective.
\end{lemma}

\newcommand{\hgamma}{\rule{0ex}{3ex}\widehat{\gamma}}
\newcommand{\hxi}{\widehat{\xi}}
\newcommand{\hcalx}{\widehat{\mathcal{X}}}
\newcommand{\calx}{\mathcal{X}}


\begin{proof}
The injectivity is clear. We then show the multiplicativity. 
Let $\hgaml \halpha \hgaml$ and $\hgaml \hbeta \hgaml$ be elements of $\hrl$. 
Put \eqna{
	\hcalx_{\hxi} &=& 
	\left\{ 
			\hgaml \hgamma \in \hgaml \bs \hgaml \hbeta \hgaml \ 
			| \ \hxi \hgamma^{-1} \in \hgaml 	\halpha \hgaml 
		\right\}
	\quad \text{for each $\hxi \in \hdl$}, 	
\\
	\calx_\xi &=&
	\left\{ 
			\gaml \gamma \in \gaml \bs \etal^{-1}(\hgaml \hbeta \hgaml) \ 
			| \ \xi \gamma^{-1} \in \etal^{-1}(\hgaml \halpha \hgaml) 
		\right\}
	\quad \text{for each $\xi \in \dl$}.
}
By \cite[Propositon 2.4]{Hy}, the product $(\hgaml \halpha \hgaml) \cdot (\hgaml \hbeta \hgaml)$ equals 
\[
\sum_{\hgaml \hxi \hgaml \in \hgaml \bs \hdl / \hgaml} 
	\left| 
	\hcalx_{\hxi}
	\right| 
\hgaml \hxi \hgaml.
\]
Note that $|\hcalx_{\hxi} | = |\hcalx_{\hxi'}|$ if $\hgaml\hxi \hgaml = \hgaml\hxi' \hgaml $,
then its image under $\etaul$ is equal to
\[
\sum_{\gaml \xi \gaml \in \gaml \bs \dl / \gaml} 
	\left| \hcalx_{\etal(\xi)}\right| 
\gaml \xi \gaml.
\]
On the other hand, the product $\etaul(\hgaml \halpha \hgaml) \cdot \etaul(\hgaml \hbeta \hgaml)$ is equal to
\[
\sum_{\gaml  \xi \gaml \in \gaml \bs \dl / \gaml} 
	\left| \calx_{\xi}\right| 
\gaml \xi \gaml.
\]
By assumption, $\calx_{\xi}$ and $\hcalx_{\etal(\xi)}$ have the same cardinality for each $\xi \in \dl$, which completes the proof. 
\end{proof}
\medskip

%
%


\medskip

For instance, we prove the bijectivity of $\etadl$ in the case where $L$ is the ring $\Z^r$ of the direct sum of $r$-copies of $\Z$.

\begin{lemma}\label{lemma:etadZr:bijective}
The map $\etadZr$ is bijective.
\end{lemma}

\newcommand{\Mr}{\mathcal{M}_r}
\newcommand{\hMr}{\rule{0ex}{3ex}\widehat{\mathcal{M}}_r}
\newcommand{\Azr}{\mathcal{A}_{\Z^r}}
\newcommand{\hAzr}{\rule{0ex}{3ex}\widehat{\mathcal{A}}_{\Z^r}}

\begin{proof}
Let $\Mr$ $(resp. \ \hMr )$ be the set of $\Z$ $(resp. \ \hZ)$ submodules of $\Z^r$ $(resp. \ \hZ^r)$ of finite index.
Then, they are naturally identified with $\gamZr \bs \dZr$ and $\hgamZr \bs \hdZr$, respectively. Moreover, the map $\etadZr$ is nothing but the map $\Mr \to \hMr$ being $(M \mapsto M \otimes \hZ)$, whose inverse map is given by $(\rule{0ex}{3ex}\widehat{M} \mapsto \widehat{M} \cap \Z^r)$.
This completes the proof.
\end{proof}
\medskip

%

\newcommand{\gzt}{ G_{\Z^2}}
\newcommand{\dzt}{ \Delta_{\Z^2}}
\newcommand{\gamzt}{ \Gamma_{\Z^2}}
\newcommand{\hgamzt}{\rule{0ex}{3ex}\widehat{\Gamma}_{\Z^2}}
\newcommand{\hdzt}{\rule{0ex}{3ex}\widehat{\Delta}_{\Z^2}}
Let $\Hei$ denote the Heisenberg Lie algebra. In the rest of this section, we study the maps $\etadH$ and $\etauH$. To begin with, let us recall some results in \cite[Section 3]{Hy}. 
Let us regard $\gzt$, $\dzt$ and $\gamzt$ as subsets of $GL_2(\Q)$, $M_2(\Z)$ and $GL_2(\Z)$, respectively. Then, $\gH$ is identified with the following group.
\begin{enumerate}
\item The underlying set is $\gzt \times \Q^2$.
\item For any two elements $(A,\va)$ and $(B,\vb)$ of $\gzt \times \Q^2$, their product $(A,\va)(B,\vb)$ is defined to be $\left(AB, A\vb + |B|\va\right)$, where $\va$ and $\vb$ are column vectors, and $|B|$ is the determinant of $B$.
\end{enumerate}
Clearly, we have $\dH = \dzt \times \Z^2$ and $\gamH = \gamzt \times \Z^2$.
Furthermore, a similar argument shows that $\hdH = \hdzt \times \hZ^2$ and $\hgamH = \hgamzt \times \hZ^2$. Certainly, $\etaH$ is the canonical inclusion $\dzt \times \Z^2 \subset \hdzt \times \hZ^2$.
\medskip

Now, we prove the bijectivity of $\etadH$.

\begin{lemma}\label{lemma:etadHbijectivity}
$\etadH$ is bijective.
\end{lemma}
\begin{proof}
\newcommand{\mcd}{\mathcal{D}}
\newcommand{\hmcd}{\rule{0ex}{3ex}\widehat{\mathcal{D}}}
\newcommand{\mcz}{\mathcal{Z}}
\newcommand{\hmcz}{\rule{0ex}{3ex}\widehat{\mathcal{Z}}}

Let $\mcd$ (resp. $\hmcd$) be a complete system of representatives of $\gamzt\bs\dzt$ (resp. $\hgamzt\bs\hdzt$), and let $\mcz_A$ (resp. $\hmcz_A$) be a complete system of representatives of $\Z^2/|A|\Z^2$ (resp. $\hZ^2/|A|\hZ^2)$ for each $A \in \dzt$ (resp. $A \in \hdzt$). Then, by \cite[Lemma 3.2]{Hy}, the sets 
\[\{(A,\va) \ |\  A \in \mcd,\ \va \in \mcz_A\}\ \text{and} \ \{(A,\va) \ |\  A \in \hmcd,\ \va \in \hmcz_A\}\]
are complete systems of representatives of $\gamH \bs \dH$ and $\hgamH \bs \hdH$, respectively.
Lemma \ref{lemma:etadZr:bijective} implies that the natural map from $\gamzt\bs\dzt$ to $\hgamzt\bs\hdzt$ is bijective, which means that one can choose $\hmcd$ contained in $\dzt$. 
For an element $A$ of $\dzt$, $\Z^2/|A|\Z^2$ is isomorphic to $\hZ^2/|A|\hZ^2$, and thus one can choose $\hmcz_A$ contained in $\Z^2$ as well.
This completes the proof. 
\end{proof}
\medskip

By Lemmas \ref{lemma:bijectivityandmultiplicative} and \ref{lemma:etadHbijectivity}, we obtain our desired property.

\begin{prop}\label{prop:multiplicativityetauH}
The map $\etauH$ is an injective ring homomorphism.
\end{prop}
\medskip

Therefore, the proposition above and Proposition \ref{prop:structureofhrl} together imply 
our main theorem:

\begin{thm}\label{thm:assertion:onetwo}
The following assertions hold.
\begin{enumerate}
\item For each $p$, the local Hecke ring $\rHp$ is embedded into $\rH$ by $\etauH$.\label{assertion:embedding}
\item The local Hecke rings $\rHp$ commute with each other in $\rH$. \label{assertion:commutativity}
\end{enumerate}
\end{thm}

\section{The nonsurjectivity of $\etauH$}\label{section:nonsurjectivity}\quad 
\newcommand{\tmp}{\gamH \bs \etaH^{-1}(\hgamH \halpha \hgamH)/\gamH}
In this section, we show the nonsurjectivity of  the map $\etauH:\hrH \to \rH$. To prove this, it is obviously sufficient to check that $\tmp$ consists of more than one elements for some $\halpha \in \hdH$. Our aim of this section is to show the following theorem.

\begin{thm}
For each prime number $p$, there exists an element $\halpha$ of $\dHp$ such that $\tmp$ consists of more than one elements.
\end{thm}

\newcommand{\gHq}{G_{\mathcal{H}_q}}
\newcommand{\dHq}{\Delta_{\mathcal{H}_q}}
\newcommand{\gamHq}{\Gamma_{\mathcal{H}_q}}
Let us fix a prime number $p$ and an element $\halpha$ of $\dHp$. 
Since $\halpha \in \dHp$, we have
\[\hgamH \halpha \hgamH = \gamHp \halpha \gamHp \times \prod_{q\not = p} \gamHq.\]
Hence, it follows from \cite[Proposition 5.1]{Hy} that, to study the set $\tmp$, it is sufficient to consider the case where 
\[
	\halpha =  \left(\diag[p^{l},p^{l+k}],
	  \begin{pmatrix}
	p^j  \nonumber \\
	p^{i+j}   \\ 
	\end{pmatrix}\right)\ 
	\text{with $0 \leq j \leq l$ and $0 \leq i \leq k$}. 
\] 
Then, we put \[A = \diag[p^{l},p^{l+k}], \quad \va = \begin{pmatrix}p^j  \nonumber \\p^{i+j}\end{pmatrix}.\]

\newcommand{\tmpa}{\etaH^{-1}(\hgamH\halpha\hgamH)}
Denote by $\gzpk$ be the subgroup of $\Gamma_{\Z_p^2}$ consisting of the elements whose $(2,1)$ entries are divided by $p^k$, and 
let $A\Z^2$ be the subgroup of $\Z^2$ consisting of the products of the matrix $A$ and column vectors with two entries of integers.
We first prove the following lemma.
\begin{lemma}
The following assertions hold.
\begin{enumerate}
\item For each element $(B,\vb)$ of $\tmpa$, we have $B \in \Gamma_{\Z^2} A \Gamma_{\Z^2}$.
\item For each element $\vb$ of $\Z^2$, $(A,\vb) \in \tmpa$ if and only if $\vb \equiv X\va \mod A\Z_p^2$ for some $X \in \gzpk$.
\item For any two elements $\vb$ and $\vc$ of $\Z^2$, $\gamH(A,\vb)\gamH = \gamH(A,\vc)\gamH$ if and only if $\vb \equiv X\vc \mod A\Z^2$ for some $X \in \gzpk \cap \Gamma_{\Z^2}$.
\end{enumerate}
\end{lemma}

\begin{proof}
For each prime number $q$, let us regard $\dHq$ as a subset of $\hdzt \times \hZ^2$. Then, we have 
\[\tmpa = \gamHp \halpha \gamHp \cap \left(\cap_{q\not = p} \gamHq\right).\]
By \cite[Proposition 3.3]{Hy}, we obtain \[B \in \Gamma_{\Z_p^2} A \Gamma_{\Z_p^2} \cap \left(\cap_{q \not =  p} \Gamma_{\Z_q^2}\right).\]
On the other hand, there exist positive integers $x$ and $y$ such that $x$ divides $y$ and $\diag[x,y]$ is contained in $\Gamma_{\Z^2}B\Gamma_{\Z^2}$. Since $\diag[x,y] \in \cap_{q\not = p} \gamHq$, it follows that $x$ and $y$ are powers of $p$. As $\diag[x,y] \in \Gamma_{\Z_p^2} A \Gamma_{\Z_p^2}$, we have $\diag[x,y] = A$, which completes the proof of the first assertion.

The second and the third assertions are a straightforward consequence of \cite[Proposition 3.3]{Hy}.
\end{proof}
\medskip
Note that $\Z^2/A\Z^2$ is isomorphic to $\Z_p^2/A\Z_p^2$. Since $\gzpk$  and $\gzpk \cap  \Gamma_{\Z^2}$ act on $\Z^2/A\Z^2$ in a natural way, the lemma above implies the following equality.

\begin{cor}\label{cor:cardinalityorbit} The following equality holds.
\[\left|\tmp\right| = \left| \gzpk \cap \Gamma_{\Z^2} \bs \ \gzpk (\va \mod A\Z^2)\right|,\]
where $\gzpk (\va \mod A\Z^2)$ is the $\gzpk$ orbit of $(\va \mod A\Z^2)$ in $\Z^2/A\Z^2$.
\end{cor}
\medskip

Next, we compute the right-hand side of the equality above.
Let us denote by $\gplpk$ the set \[\left\{ \Imat + \diag[p^l,p^{l+k}]M_2(\Z_p) \right\} \cap \Gamma_{\Z_p^2},\] then we see easily that $\gplpk$ coincides with the set of elements fixing every element of $\Z^2/A\Z^2$. Especially, it is a normal subgroup of $\gzpk$.
The quotient group $\gzpk / \gplpk$ is written by $\glk$. Note that it also acts on $\Z^2/A\Z^2$.
\newcommand{\detlk}{\det_{l,k}}
Let us put $U_0 = \Z_p^*$ and $U_i = 1+p^i\Z_p$ for each $i >0$. Since $\det(\gplpk) = U_l$, the determinant map induces the homomorphism $\detlk:\glk \to U_0/U_l$. Put $\glko = \ker \detlk$ and $\glkpm = \detlk^{-1}\{\pm 1\}$. We prove the following key proposition. 

\begin{prop}\label{prop:surjectivity}
The canonical homomorphisms $\gzpk \cap \gamzt \to \glkpm$ and $\gzpk \cap SL_2(\Z) \to \glko$ are surjective.
\end{prop}
\medskip

To prove this, we prepare some lemmas.

\begin{lemma}
Let $c$, $d$ and $n$ be integers with $c\not = 0$ and $(d,n)=1$. Then, there exists an integer $x$ such that $(c, d+nx)=1$. 
\end{lemma}

\begin{proof}
Since $c \not = 0$, the set $S$ of prime numbers which divide $c$ and not $d$ is a finite set.
Let $x$ be the product of all elements of $S$, then $x$ has our desired property. 
\end{proof}

The following lemma slightly refines the well known fact that $SL_2(\Z) \to SL_2(\Z/n\Z)$ is surjective for each positive integer $n$.

\begin{lemma}
Let $n$ be a positive integer, and let $\pmat{a}{b}{c}{d}$ be an element of $M_2(\Z)$ whose determinant is congruent to $1$  
modulo $n$. If $c \not = 0$ and $(d,n)=1$, then there exist integers $a'$, $b'$ and $d'$ satisfying that they are respectively congruent to $a$, $b$ and $d$ modulo $n$, and that $\pmat{a'}{b'}{c}{d'}$ is an element of $SL_2(\Z)$. 
\end{lemma}

\begin{proof}
By the previous lemma, we choose $x \in \Z$ such that $(c, nx+d) = 1$, and then choose $y,\ z\in \Z$ satisfying $cy + (nx+d)z = 1$. By assumption, there exists an integer $r$ such that \[\det\pmat{a}{b}{c}{nx+d} = 1-  nr.\] Thus $a'= a+nrz$, $b' = b - nry$ and $d' = d + nx$ have the desired property.
\end{proof}
\medskip

\newcommand{\sgam}{\widetilde{S}_{\va}}
\newcommand{\sglk}{
S_\va
}
Now we give a proof of Proposition \ref{prop:surjectivity}.
\begin{proof}[Proof of Proposition \ref{prop:surjectivity}]
For any two elements $X$ and $Y$ of $\gzpk$, $X^{-1}Y$ belongs to $\gplpk$ if and only if $X - Y$ is an element of $\diag[p^l,p^{l+k}]M_2(\Z_p)$.
Hence, we may take a complete system of representatives of $\glk$ contained in $M_2(\Z)$.
Thus, the surjectivity of $\gzpk \cap SL_2(\Z) \to \glko$ is an easy consequence of the lemma above.
Since the matrix $\diag[1,-1]$ is contained in $\gzpk\cap \gamzt$ and its determinant equals $-1$, the natural map $\gzpk \cap \gamzt \to \glkpm$ is surjective as well,
which completes the proof. 
\end{proof}
\medskip

By Corollary \ref{cor:cardinalityorbit} and Proposition \ref{prop:surjectivity}, we therefore have
\[\left|\tmp\right| = \left|\glkpm \bs \glk(\va \mod A\Z^2)\right|.\]

Let us denote by $\sglk$ the stabilizer subgroup of $(\va \mod A\Z^2)$ in $\glk$. Since $\detlk$ induces the isomorphism $\glk/\glko \to U_0/U_l$, the right-hand side of the equality above equals $\left| U_0/ \pm {\detlk} (\sglk) \right|$. Hence we are reduced to calculating $\detlk(\sglk)$.

\begin{prop}\label{prop:detsa}
Put $n = min(i, k-i, l-j)$. Then we have $\detlk(\sglk) = U_n/U_l$.
\end{prop}

\begin{proof}
Let $\pmat{a}{b}{p^kc}{d}$ be an element of $\gzpk$. Then it belongs to the stabilizer subgroup $\sgam$ of $(\va \mod A\Z^2)$ in $\gzpk$ if and only if there exist elements $x$, $y$ of $\Z_p$ such that $a =1 - bp^i + p^{l-j}x$ and $d = 1 - p^{k-i}c + p^{l-j+k-i}y$. By a direct computation, we complete the proof.
\end{proof}
\medskip

Our aim of this section is settled by the following corollary.

\begin{cor} \label{cor:explicitcaluculation} The following equality holds.
\[\left| \tmp \right|= [U_0:\pm U_n].\] 
Especially, we have
\eqna{
 \left| \tmp \right| =
 \left\{
 \begin{matrix}
  p(p-1)/2 & \text{if $p\not = 2$, $l=2$,$k=4$, $j=0$, $i=2$},\\
  2 & \text{if $p = 2$, $l=3$,$k=6$, $j=0$, $i=3$}.
 \end{matrix}
 \right.
 }
\end{cor}
\begin{proof}
The proof is an immediate consequence of Proposition \ref{prop:detsa}.
\end{proof}
\medskip

In the conclusion of this paper, we obtain the following theorem.

\begin{thm}\label{thm:main}
The following assertions hold.
\begin{enumerate}
\item For each $p$, the local Hecke ring $\rHp$ is embedded into $\rH$ by $\etauH$.
\item The local Hecke rings $\rHp$ commute with each other in $\rH$.
\item $\rH$ is not generated by the local Hecke rings $\rHp$ as a ring.
\end{enumerate}
\end{thm}
\begin{proof}
Assertions $1$ and $2$ have been proved in Theorem \ref{thm:assertion:onetwo}.
Proposition \ref{prop:structureofhrl} and the nonsurjectivity of $\etauH$ imply assertion 3.
\end{proof}

\section*{Acknowledgement}

The author would like to thank Professor Hiroaki Nakamura for his continuous support, many suggestions and warm encouragement over the years. The author also appreciates the unknown referee's valuable and profound comments.

\textsc{Department of Health Informatics, Faculty of Health and Welfare Services Administration, Kawasaki University of Medical Welfare, Kurashiki, 701-0193, Japan}

{\it Email address}: \texttt{fumitake.hyodo@mw.kawasaki-m.ac.jp}

\end{document}